\documentclass[11pt, a4paper]{amsart} % A4 format
\usepackage[english]{babel}
\usepackage{amsmath,amsfonts}
\usepackage{amsthm}
\usepackage{amssymb}
\usepackage[alphabetic]{amsrefs}
\usepackage{url}
\usepackage[all]{xy}
\usepackage{graphicx}
\usepackage{psfrag}
\usepackage{array}
\usepackage{enumerate}
\usepackage{mathtools}
\usepackage{color} %
\usepackage{dsfont} 
\usepackage{mathrsfs}
\usepackage{lipsum}
\usepackage{relsize} %for \mathlarger{} 
\usepackage{upgreek} %for \uptau, etc.
\usepackage{caption}
\usepackage{pgf,tikz}
\usetikzlibrary{arrows}
\usepackage{xargs} %For xcommand and other stuff

\usepackage{xcolor}
\definecolor{rouge}{rgb}{0.85,0.1,.4}
\definecolor{bleu}{rgb}{0.1,0.2,0.9}
\definecolor{violet}{rgb}{0.7,0,0.8}
\definecolor{noir}{rgb}{0,0,0}
\usepackage[colorlinks=true,linkcolor=bleu,urlcolor=noir,citecolor=rouge]{hyperref} 

\theoremstyle{theorem}
\newtheorem{theorem}{Theorem}[section]

\newtheorem{proposition}[theorem]{Proposition}

\newtheorem{lemma}[theorem]{Lemma}

\theoremstyle{definition}
\newtheorem{definition}[theorem]{Definition}
\theoremstyle{definition}
\newtheorem{remark}[theorem]{Remark}

\newtheorem{example}[theorem]{Example}

\newcommand{\normal}{\vartriangleleft} % 
\newcommand{\bra}{\langle}
\newcommand{\ket}{\rangle}

\newcommand{\wobbling}[1]{\mathscr{W}(#1)}
\newcommand{\borel}[1]{\mathscr{B}(#1)}
\newcommand{\clopen}[1]{\mathscr{C}(#1)}
\renewcommand{\setminus}{\smallsetminus}

\newcommand{\Z}{\mathbf{Z}}

\newcommand{\R}{\mathbf{R}}

\begin{document}
\title{The Burnside Problem for Locally Compact Groups}% valued in the Steinberg representation}

%%% Author 
\author{Thibaut Dumont}
\thanks{Thibaut Dumont has been supported by the Academy of Finland (grant 288501 `\emph{Geometry of subRiemannian groups}') and the European Research Council (ERC Starting Grant 713998 GeoMeG `\emph{Geometry of Metric Groups}'), P.I.\ Enrico Le Donne.}
\address{Department of Mathematics and Statistics, P.O. Box 35 (MaD)\\ FI-40014 University of Jyv\"askyl\"a, Finland}
\email{thibaut.dumont@math.ch}

\author{Thibault Pillon}
\thanks{Thibault Pillon has been supported by the European Research Council consolidator grant 614195.}
\address{Analysis Section, KU Leuven, Celestijnenlaan 200b, Box 2400\\ 3001 Leuven, Beligum}
\email{thibault.pillon@kuleuven.be}

%%%% Abstract %%%%
\begin{abstract}
 Using topological notions of translation-like actions introduced by Schneider, we give a positive answer to a geometric version of Burnside problem for locally compact group. The main theorem states that a locally compact group is non-compact if and only if it admits a translation-like action by the group of integers $\Z$. We then characterize the existence of cocompact translation-like actions of $\Z$ or non-abelian free groups on a large class of locally compact groups, improving on Schneider's results and generalising Seward's.
\end{abstract}

%\date{\today}
\renewcommand{\keywordsname}{Key words}
 %Keywords
%\subjclass[2010]{AAXXBB, CCYYDD} %2010 MSC

\maketitle
\setcounter{tocdepth}{1} %Depth of the table of content
{ %This ensures the table of content to stay black even if its entries are links
%\hypersetup{linkcolor=black}
%\tableofcontents
}

%%%%%%%%%%% INTRODUCTION %%%%%%%%%%%%
\section{Introduction}

Several problems in geometric group theory ask whether properties of a group can be formulated in terms of subgroup containment. The Day-von Neumann problem asks whether a group is non-amenable if and only if it contains a non-abelian free subgroup. The general Burnside problem, which is the simplest form of the problem, asks whether a finitely generated group is infinite if and only if it contains infinite cyclic subgroups. The so-called Gersten conjecture, formulated by Gromov, loosely asks whether any non-hyperbolic group contains a Baumslag-Solitar group $\operatorname{BS}(m,n)$. Similarly, a sufficient condition for a group to have exponential growth is to possess a free subsemigroup and it is very natural to wonder if the converse also holds.

However, all these questions have been answered by the negative. Ol'shanskii \cite{Olshanskii} proved the existence  of non-amenable groups not containing free subgroups. Together with Sapir \cite{OlshanskiiSapir}, they extended this result to finitely presented groups and later Monod \cite{Monod} provided the first explicit examples. Golod and Shafarevich \cite{GolodSafarevic} provided the first counter-examples to the general Burnside problem and more examples, such as Grigorchuk's group of intermediate growth are now very well understood. It follows from the work of Brady \cite{Brady} concerning non-hyperbolic finitely presented subgroups of hyperbolic groups that the Gersten conjecture is false.
Finally, it is part of folklore in the field that there exist torsion groups of exponential growth, for example a wreath product of the Grigorchuk group with $\Z/2\Z$ does the job.

Despite these negative answers, a lot of positive answers were obtained in essentially two ways, restricting to a smaller class of groups or weakening the notion of subgroup containment. The famous Tits alternative gives a positive and strong solution to the Day-von Neumann problem for the class of finitely generated linear groups and has been extended to many classes of groups included e.g. mapping class groups and subgroups of $\operatorname{Out}(F_n)$ \cite{Tits,Ivanov,McCarthy,BestvinaFeighnHandel}. Rickert~\cite{Rickert} already observed that the Day-von Neumann problem holds in the class of almost connected groups.
Gersten's conjecture was actually named in reference to Gersten's theorem \cite{Gersten} showing that hyperbolicity pass to finitely presented subgroups of groups of cohomological dimension 2. In particular they cannot contain Baumslag--Solitar groups. 
 Finally, Chou \cite{Chou} proved that elementary amenable groups have exponential growth if and only if they contain free subsemigroups.

A positive answer to a measure theoretic weakening of the Day--von Neumann problem arised from the work of Gaboriau and Lyons \cite{GaboriauLyons} and was generalized to locally compact groups by Gheysens and Monod \cite{GheysensMonod}.

Our main concern in this article are the so-called geometric versions of the Day-Von Neumann and Burnside problems. Whyte \cite{Whyte} introduced the notion of translation-like actions as a weak notion of subgroup containment.

\begin{definition} A free action of a group $G$ on a metric space $X$ is called \emph{translation-like} if for every $g\in G$ there exists $C\ge0$ such that $d_X(x,gx)\le C$ for all $x\in X$.
\end{definition}

Permutations satisfying the condition in this definition are called \emph{wobbling transformations} of $X$ and form a group $\wobbling{X}$. Groups of wobbling transformations or \emph{wobbling groups} have attracted a lot of interest in recent works (see e.g. \cite{MonodJuschenko,JuschenkoDeLaSalle}). Whyte used this notion to prove the following geometric analogue of the Day-von Neumann statement. 

\begin{theorem}[Whyte {\cite[Theorem 6.1]{Whyte}}]
A finitely generated group $G$ is non-amenable if and only if it admits a translation-like action of  the free group on $d$ generators $F_d$ for some (equivalently all) $d\ge2$.
\end{theorem}
He then asked whether the Burnside problem and the Gersten conjecture admit posi\-
tive answers in this setting. Seward \cite{Seward} came up with a solution to first problem.

\begin{theorem}[Seward {\cite[Theorem 1.4]{Seward}}]\label{Seward}
A finitely generated group $G$ is infinite if and only if it admits a translation-like action of $\Z$.
\end{theorem}

He was also able to characterize the existence of transitive translation-like actions of both $\Z$ and $F_d$.

\begin{theorem}[Seward {\cite[Theorem 1.5]{Seward}}]\label{SewardTransitive}
Let $G$ be a finitely generated infinite group. Then:
\begin{enumerate}
\item $G$ admits a transitive translation-like action by $\Z$ if and only if $G$ is one-ended or two-ended.
\item $G$ admits a transitive translation-like action by $F_d$ for every $d\ge 2$ if and only if $G$ is not amenable.
\end{enumerate}
\end{theorem}

However, counter-examples to both directions of the geometric Gersten conjecture were exhibited by Jiang~\cite{Jiang} and Cohen~\cite{Cohen}.

Recently, Schneider defined two suitable notions of translation-like actions for actions on locally compact groups (see Section \ref{Preliminaries} for a precise definition). He was able to prove the following generalization of Whyte's result.

\begin{theorem}[Schneider {\cite[Theorem 1.4]{Schneider}}]\label{Schneider}
Let $G$ be a locally compact group. The following are equivalent.
\begin{enumerate}
\item $G$ is not amenable.
\item $G$ admits a Borel (or equivalently clopen) translation-like action by  $F_d$ for all $d\ge2$.
%\item $G$ admits a clopen translation-like action by  $F_d$ for all $d\ge2$.
\end{enumerate}
\end{theorem}

In the present paper, we use Schneider's notion to give a positive answer to the geometric Burnside problem for locally compact groups.

\begin{theorem}[Main Theorem]\label{ThmMain}
Let $G$ be a compactly generated locally compact group. The following are equivalent:
\begin{enumerate}
\item $G$ is non-compact.
\item $G$ admits a Borel (or equivalently clopen) translation-like action by $\Z$.
%\item $G$ admits a clopen translation-like action by $\Z$.
\end{enumerate}
\end{theorem}

The proof relies on the fact that non compact almost-connected groups always contain discrete infinite cyclic subgroups. It serves a similar purpose in our proof as Rickert's result in Schneider's proof of Theorem \ref{Schneider}. In conjonction with van Dantzig's theorem, we get the following proposition %Our proof strongly relies on the following result for locally compact groups.

 \begin{proposition}\label{Dichotomy} Let $G$ be a locally compact group. At least one of the following holds:
\begin{enumerate}[(i)]
\item $G$ contains a compact open subgroup.
\item $G$ contains a discrete infinite cyclic subgroup.
\end{enumerate}
\end{proposition}

%Although this result seems to be already present in the minds of experts on locally compact groups, it was absent in the literature. %we were unable to find an actual reference to it in the literature.  
We provide a proof essentially observing that all necessary ingredients are already present in the work of Rickert \cite{Rickert}. A proof by Gaillard and Karai is discussed on a MathOverflow entry ( \cite{Gaillard}).

Compact generation in Theorem \ref{ThmMain} is a natural assumption generalising finite generation in Seward's result, however the theorem admits a more general statement to arbitrary locally compact groups.  A locally compact group is \emph{locally elliptic} if every compact subset is contained in a compact subgroup (see Section 4.D of \cite{CornulierdelaHarpe}). In the case of discrete groups, this definition coincide with the classical notion of local finiteness. We get the following statement:

\begin{theorem}\label{ThmLocallyElliptic}
Let $G$ be a locally compact group. The following are equivalent:
\begin{enumerate}
\item $G$ is not locally elliptic.
\item $G$ admits a Borel (or equivalently clopen) translation-like action by $\Z$.
%\item $G$ admits a clopen translation-like action by $\Z$.
\end{enumerate}
\end{theorem}

Finally we are able to generalise Theorem \ref{SewardTransitive} to a large class of groups using cocompactness as an analogue of transitivity. 

\begin{theorem}\label{SewardTransitiveLC}
Let $G$ be a compactly generated non-compact locally compact group. Assume moreover that $G$ admits a compact open subgroup. Then the following hold.
\begin{enumerate}
\item $G$ admits a cocompact Borel (or equivalently clopen) translation-like action by $\Z$ if and only if $G$ is one-ended or two-ended.
\item $G$ admits a cocompact Borel (or equivalently clopen) translation-like action by $F_d$ for every $d\ge 2$ if and only if $G$ is not amenable.
\end{enumerate}
\end{theorem}

%%%%%%%%%%% ACKNOWLEDGMENTS %%%%%%%%%%
\medskip 

\noindent \textbf{Acknowledgments.} The authors would like to thank Pierre-Emmanuel Caprace, Nicolas Monod, Stefaan Vaes and George Willis for interesting discussions. They also thank Colin Reid for a decisive discussion about local ellipticity. 

The first named author also benefited from the hospitality and knowledge of Michael Cowling during a stay at the University of New South Wales, Australia.

%The first author was partially supported by the Academy of Finland
%(grant 288501 `\emph{Geometry of subRiemannian groups}')
%and by the European Research Council
%(ERC Starting Grant 713998 GeoMeG `\emph{Geometry of Metric Groups}')

%%%%%%%%%%% PRELIMINARIES %%%%%%%%%%%%
\section{Preliminaries on LC-groups and TL-actions}\label{Preliminaries}

By convention, locally compact groups are Hausdorff and abreviated LC-groups.
\subsection{Translation-like actions}
\begin{definition}\label{def:metricwobbling} Let $X$ be a metric space. Recall from the introduction that the \emph{wobbling group} of $X$, denoted by $\wobbling{X}$, is the group of all bijections $\varphi$ of $X$ satisfying 
\[ \exists C>0, \,\forall x\in X,\, d_X(x,\varphi(x))\le C.\]

An action of a group $G$ on $X$ is called \emph{translation-like} (TL-action below) if the group acts freely by wobbling transformations.

\end{definition}

The following example motivates the terminology of translation-like and establishes the idea that TL-actions generalize subgroup containment.

\begin{example} Let $G$ be a group equipped with a left-invariant metric $d$. Right translations of the form $x\mapsto xg$ are wobbling transformations. Indeed, $d(x,xg)=d(1,g)$ is constant.

It follows that the right action of any subgroup $H<G$ is a TL-action.
\end{example}

Although locally compact groups can be equipped with metrics which behave nicely with respect to their topology, generalising Definition \ref{def:metricwobbling} in a straightforward way would not give a good notion of TL-action for these groups. Instead of defining them in term of a metric, Schneider first made the following observation:

\begin{example} Let $G$ be a group equipped with a proper (i.e. such that all balls are finite) left-invariant metric. This is in particular the case when $G$ is finitely generated and equipped with the word-metric associated to a finite generating set, but such metrics exist on any countable group. Then for any wobbling transformation $\varphi\in\wobbling{G}$, there exist a finite subset $F\subset G$ such that
\[\varphi(x)\in xF,\, \forall x\in G.\] 

In other words, there exist a partition $\{P_g\}_{g\in F}$ of $G$ such that for all $g\in F$, $\varphi(x)=xg$ for all $x\in P_g$. 
\end{example}

This observation allows to see wobbling transformation as piecewise-(right )translations and serves as a motivation for the following definitions:

\begin{definition}\label{Def:Piecewise-translation} Let $G$ be a locally compact group and let $\alpha:G\to G$ be a bijection such that there exists a finite partition $\mathscr P$ of $G$ with
\[\forall P\in\mathscr P, \exists g_P: \forall x\in P, \alpha(x)=xg_P.\]
\begin{enumerate}
\item If $\mathscr P$ is a partition of $G$ into Borel subsets, we call $\alpha$ a \emph{Borel piecewise-translation} and we denote by  $\borel{G}$ the group of all such bijections.
\item If $\mathscr P$ is a partition of $G$ into clopen subsets, we call $\alpha$ a \emph{clopen piecewise-translation} and we denote by  $\clopen{G}$ the group of all such bijections.
\item If $\varphi:H\to\borel{G}$ (resp. $\clopen{G}$) is a free action of $H$ onto $G$ which admits a measurable fundamental domain, we call $\varphi$ a \emph{Borel (resp. clopen) translation-like action}.
\end{enumerate}
\end{definition}
\begin{remark} Observe that when $G$ is endowed with a right-invariant Haar measure, any Borel piecewise-translation is a measure preserving transformation. Similarly, any clopen piecewise translation is a homeomorphism of $G$ and in the particular case of $G$ being connected, we actually have that $\clopen{G}\cong G$. 
\end{remark}
\begin{remark} Beware that contrary to Schneider, we define piecewise-translations using right translations. This is purely a matter of conventions, but we chose to be in accordance with Whyte's and Seward's choice. Accordingly, we need to use a right-invariant Haar measure. 
\end{remark}

A result of Feldman and Greenleaf (\cite{FeldmanGreenleaf}) asserts that any closed subgroup of a LC-group admits a measurable cross-section. This shows that the right action of a closed subgroup $H$ of $G$ induces a clopen TL-action of $H$ onto $G$. Without requiring the existence of a measurable fundamental domain in the definition of a Borel (or clopen) translation-like action, this would be true of any subgroup of $G$ and Theorem \ref{Schneider} or Theorem \ref{ThmMain} would fail due to the existence of elements of infinite order and free subgroups in compact (hence topologically amenable) groups. 

%%%%%%% The almost-connected case %%%%%%%
\subsection{The almost connected case}
This paragraph treats the case of an \emph{almost connected} LC-group $G$, that is a connected-by-compact LC-group. The arguments presented below go back to Iwasawa \cite{Iwasawa}, Rickert \cite{Rickert, Rickert2}, and the solution to Hilbert's fifth problem. The main ingredients are essentially the Levi decomposition and the fact that a connected semisimple Lie group is amenable if and only if it is compact.

\begin{theorem}[Almost connected case] \label{thm:almostconnected} Let $G$ be an almost connected LC-group. If $G$ is non-compact, it contains a discrete infinite cyclic subgroup. 
\end{theorem}
\begin{proof}Suppose $G$ is non-compact. By the structure of locally compact groups, see for instance \cite[Theorem 2.E.14]{CornulierdelaHarpe}, there is a compact normal subgroup $K<G$ such that $G/K$ is a Lie group with finitely many connected components. Therefore, we may assume that $G$ is a connected Lie group. Let $R$ denote the (solvable) radical of $G$, that is the maximal solvable normal subgroup of $G$. It is a closed subgroup of $G$ and, according to \cite[Theorem 5.3]{Rickert2}, $G$ is amenable if and only if $G/R$ is compact. In fact, non-amenability of $G$ is equivalent to the existence of a discrete free subgroup $F_2<G$ by \cite[Theorem 5.5]{Rickert2}. So we may assume that $G/R$ is compact (in which case $G$ is called a (C)-group~\cite{Iwasawa}). Since $R$ is not compact by assumption, it has a subquotient $H/N$ isomorphic to~$\R$ by \cite[Theorem 19]{Iwasawa}, where $N\normal H<G$ are closed subgroups. Therefore $H/N$ has a discrete infinite cyclic subgroup which lifts to a discrete $\Z<H<G$ as desired.
\end{proof}

\begin{proof}[Proof of Proposition \ref{Dichotomy}] Let $G$ be a locally compact group containing no compact open subgroup. By van Dantzig's theorem, $G$ contains an open subgroup $H<G$ which is almost connected and non-compact by assumption. The previous theorem implies that $H$ contains a discrete subgroup~$\Z$.
\end{proof}
%%%%

%%%%%%% The almost-connected case %%%%%%%
\subsection{Cayley-Abels graphs}
In the presence of compact open subgroups, compactly generated groups enjoy nice analogues of Cayley graphs, namely \emph{Cayley-Abels} graphs. This will allow us to rely on the results of Seward \cite{Seward} characterizing the existence of translation-like action of $\Z$ on infinite graphs of bounded valency. We are not interested in the precise way a compact open subgroup gives raise to a graph but rather in their existence and properties.

\begin{proposition}[{\cite[Proposition 2.E.9]{CornulierdelaHarpe}}] \label{prop:2E9} For a LC-group $G$, the following are equivalent:
\begin{enumerate}[(i)]
\item $G$ is compactly generated and has a compact open subgroup.
\item There exists a connected graph $X$ of bounded valency and a continuous action of $G$ on $X$ which is proper and vertex-transitive.
\end{enumerate}

Such a graph $X$ is called a Cayley-Abels graph of $G$.
\end{proposition}

In this situation there is a good notion of ends for $G$. Any two Cayley-Abels graphs of $G$ are quasi-isometric since they both admit proper cocompact action of $G$. We can then define the number of ends $e(G)$ of $G$ as the number of ends of its $Cayley-Abels$ graphs. With this definition and in perfect analogy to the discrete case,a group has either 0 (in which case it is compact),1,2 or infinitely many ends. When $e(G)=\infty$, an analogue of Stalling's splitting theorem is available and the group always splits as either an amalgamated free product over a compact subgroup or as an HNN-extension over a compact subgroup (see \cite[Section 4.D.]{Cornulier}).

The next proposition shows that if $G$ is unimodular, then any translation-like action of a free group on $X$ lifts to a clopen TL-action on $G$. The lifted action depends on some choices and is certainly not unique.

%
%Let $X$ be a Cayley-Abels graph of a LC-group $G$ as above, and fix $x_0\in X$ a base vertex. The isotropy subgroup $K=\{g\in G\mid gx_0=x_0\}$ is compact open and we can identify $X$ with the left cosets $G/K$ via the orbit map $g\mapsto gx_0$.
% Since $X$ has bounded valency and is homogeneous under $G$, we may find a finite subset $S\subset G$ such that $gK$ and $g'K$ are adjacent if and only if $g'K=gsK$ for some $s\in S$. 

\begin{proposition}\label{CayleyAbelsLift} Suppose that a LC-group $G$ admits a Cayley-Abels graph $X$. If $G$ is unimodular, then any translation-like action of a free group $F_d$ on $d\ge1$ generators on $X$ lifts to a clopen TL-action of $\Gamma$ onto $G$ with a clopen fundamental domain.
\end{proposition}
\begin{proof} Let $K$ denote the stabiliser of a vertex in $G$; since the action is continuous and proper, $K$ is a compact and open subgroup of~$G$. Since $X$ is homogeneous under $G$, we can identify the vertices of $X$ with the quotient $G/K$, the set of left cosets of $K$. Moreover, there exists a finite subset $S\subset G$ such that the neighbors of a coset $gK$ in $X$ are exactly $gsK$ for $s\in S$. In particular, the disjoint union $S\sqcup K$ generates $G$ and we may assume that any coset is of the form $gK$ for some $g$ in $\bra S \ket$, the subgroup generated by~$S$. Therefore, for any wobbling bijection $\alpha\in\wobbling{X}$ there exists a finite subset $T\subset S^n$ such that \[\forall g\in G,\exists s\in T: \alpha(gK)=gsK.\] Here $n$ is an integer bounding the distances $d(\alpha(gK),gK)$.

Suppose that for every $s\in T$, at the level of~$G$, there exists a clopen piecewise-(right)\-translation $\varphi_s$ mapping $K$ to $sK$ and define a map $\tilde\alpha:G\to G$ by 
\[\tilde\alpha\mid_{gK}=(\lambda_g\circ \varphi_s\circ\lambda_{g^{-1}})|_{gK} \text{ when }\alpha(gK)=gsK.\]

Observe that $\tilde\alpha$ is indeed a clopen piecewise-translation and that it factors through the quotient map $G\to G/K$ to $\alpha$.
Applying this procedure to a free generating set of $F_d$ the free action of $F_d$ on $X$ uniquely lifts to a free action on $G$. Note that both the fact that $G$ is free and that the action is free are necessary conditions for this to be true in general. Moreover, the preimage under the quotient of a fundamental domain $X_0\subset X$ is a fundamental domain for the action on $G$ and is clopen as the union of a family of cosets of the clopen subgroup $K$. The next lemma takes care of implementing  the maps $\varphi_s$.

\end{proof}

\begin{lemma}\label{Refinement} let $K$ be a compact open subgroup of a locally compact unimodular group $G$ and $s\in G$. Then there exists a (locally defined) clopen piecewise-translation $\varphi_s$ such that $\varphi_s(K)=sK$. 
\end{lemma}

\begin{proof}
Observe that $K\cap s^{-1}Ks$ and $sKs^{-1}\cap K$ are compact open subgroups of $K$, conjugated in $G$. Since $G$ is unimodular they have the same Haar measure and hence the same index in~$K$. Since $K$ is compact, the index is finite and we can write
\[K=\bigsqcup_{i=1}^n\left(K\cap s^{-1}Ks\right)r_i=\bigsqcup_{i=1}^n\left(sKs^{-1}\cap K\right)u_i\]
for some choice of representatives $r_1,\ldots,r_n,u_1,\ldots,u_n\in K$.
We get
\begin{eqnarray*}
sK&=&\bigsqcup_{i=1}^ns\left(K\cap s^{-1}Ks\right)r_i=\bigsqcup_{i=1}^n\left(sK\cap Ks\right)r_i\\
&=&\bigsqcup_{i=1}^n\left(sKs^{-1}\cap K\right)sr_i=\bigsqcup_{i=1}^n\left(sKs^{-1}\cap K\right)u_iu_i^{-1}sr_i\\
\end{eqnarray*}
Now, defining $\varphi_s$ by $\varphi_s(x)=xu_i^{-1}sr_i$ when $x\in \left(sKs^{-1}\cap K\right)u_i$, we get the desired result.
\end{proof}

%\begin{remark} \label{remark:1}In Proposition \ref{prop:lift} above, the fundamental domain $G_e$ of the lifted action is given by the lift of the fundamental domain $X_0$ in $X$. In other words, $G_e=\pi^{-1}(X_0)$ which is open and moreover compact if $X_0$ is finite.
%
% which is equivalent to $X$, respectively $G$, having finitely many ends, then $G_0$ is compact open. 
%Need for a discussion about ends.
%\end{remark}

%%%%%%%%%%% PROOF OF THE MAIN THEOREM %%%%%%%%%%%%
\section{Proof of the main theorem}

We now have all the ingredients to prove Theorem \ref{ThmMain}

\begin{proof}[Proof of Theorem \ref{ThmMain}]~\\
(2)$\Rightarrow$(1) : Since any clopen TL-action is a Borel TL-action we only need to consider the latter case. Let $\varphi\in\borel{G}$ be the Borel piecewise-translation generating the action and let $D$ be a measurable fundamental domain. Also, let $\mu$ be a right-invariant Haar measure on $G$. Since the action is free, the sets $\varphi^{(n)}(D)$, $n\in\Z$ form a partition of $G$ and since $\varphi$ is measure-preserving, we get
\[\mu(G)=\sum_{n\in\Z}\mu(\varphi^{(n)}(D))=\sum_{n\in\Z}\mu(D).\]
This leads either to $\mu(G)=0$ which is a contradiction or to $\mu(G)=\infty$ which implies that $G$ is non-compact.
 
(1)$\Rightarrow$(2) : Assume that $G$ is non-compact and recall that the existence of a discrete infinite cyclic subgroup in $G$ is a sufficient condition for the existence of a clopen TL-action of $\Z$. 

According to Proposition \ref{Dichotomy}, either $G$ has an infinite discrete cyclic subgroup, and we are done, or $G$ contains a compact open subgroup. We can therefore assume that $G$ contains a compact open subgroup.

If $G$ has infinitely many ends, the analogue of Stallings' splitting theorem tells us that $G$ splits either as an amalgamated free product over a compact subgroup or as an HNN extension over a compact subgroup. In both cases, $G$ contains free discrete subgroups and hence contains discrete copies of $\Z$. 

If $G$ has one or two ends ($G$ has at least one end since $G$ is non-compact), we consider two cases. When $G$ is not unimodular, the modular function $\Delta$ is a non-trivial continuous homomorphism from $G$ to the multiplicative group $\R_{>0}$. Any $x\in \R_{>0}$ then lifts to an element $g\in G$ generating a discrete copy of $\Z$. 

If $G$ is unimodular, consider a Cayley Abels graph $X$ associated to $G$. $X$ has either one or two ends and according to \cite[Corollary 3.4]{Seward}, $X$ admits a transitive translation-like action by $\Z$. Finally, Proposition \ref{CayleyAbelsLift} allows us to lift this action to a  clopen TL-action of $\Z$ onto $G$.

\end{proof}

We now prove corollary \ref{ThmLocallyElliptic}

\begin{proof}[Proof of Corollary \ref{ThmLocallyElliptic}]~

\noindent (2)$\Rightarrow$(1): Let $\varphi$ be a Borel piecewise-translation generating a Borel TL-action of $\Z$ on $G$ and let $F=\{x^{-1}\varphi(x)\mid x\in G\}$. By definition of a piecewise-translation, $F$ is a finite set. Now assume, for the sake of contradiction, that $G$ is locally elliptic. Let $H<G$ be a compact subgroup containing $F$. Observe that $H$ is $\varphi$-invariant and therefore $\varphi\mid_H$ defines a Borel TL-action of $\Z$ onto the compact group $H$. This is in contradiction with Theorem \ref{ThmMain}.

\noindent (1)$\Rightarrow$(2) : Suppose that $G$ is not locally elliptic and let $H$ be a non-compact, compactly generated subgroup of $G$. By taking the subgroup generated by a compact neighbourhood of the identity and $H$, we may assume that $H$ is open in $G$. Since $H$ is non-compact, there exists $\varphi\in\clopen{H}$ generating a clopen TL-action on $H$. By choosing a section $r:G/H\to G$ of the projection map $G\to G/H$, we can define an extension $\psi$ of $\varphi$ to $G$ by $\psi(g)=r(g)\varphi(r(g)^{-1}g)$. In other words, if $\mathscr P$ and $g_P,\, P\in\mathscr P$ define $\varphi$ as in Definition \ref{Def:Piecewise-translation}, then $\psi$ is defined by the partition $\mathscr Q=\{r(G/H)P\mid P\in \mathscr P\}$ of $G$ using the same translations $g_P$. We claim that $\psi\in\clopen{G}$ and that it generates a clopen TL-action of $\Z$ on $G$.

Observe that for all $P\in\mathscr P$, $r(G/H)P$ is clopen because each $r(gH)P$ is a clopen subset of the open coset $gH$ showing indeed that $\psi\in\clopen{G}$. Furthermore, if $D$ is a measurable fundamental domain of the action generated by $\varphi$, then $r(G/H)D$ is a fundamental domain of the action generated by $\psi$. It is measurable as the preimage of $D$ under the continuous map $g\mapsto r(g)^{-1}g$. Finally, it is straightforward that the action generated by $\psi$ is free.  Hence $\psi$ generates a clopen TL-action of $\Z$.
\end{proof}

\begin{remark}
In the case of a discrete group, this theorem states that a group admits a translation-like action by $\Z$ if and only if it is not locally finite. This would be very straightforward to prove from Seward's result only. Let us mention that in \cite[Proposition 4.3, 4.5 and Remark 4.6]{Schneider} Schneider already established the fact that a discrete group is not locally finite if and only if there exists a bornologous injection from $\Z$ to $G$. Theorem \ref{ThmLocallyElliptic} can be thought both as a strenghtening of this result to the existence of a translation-like action and a generalization to the topological setting.
\end{remark}

%%%%%%%%%%% COCOMPACT %%%%%%%%%%%%
\section{Cocompact TL-actions}

In this section, we prove and discuss Theorem \ref{SewardTransitiveLC}  about the existence of cocompact translation-like actions by $\Z$ and $F_n$, $n\ge2$. In our setting, cocompactness plays the role of transitivity in \cite[Theorem 1.5]{Seward}, and the class of unimodular groups possessing a compact-open subgroup seems to be the largest class in which these reults about discrete groups carry.

\begin{theorem}
Let $G$ be a non-compact unimodular compactly generated group possessing a compact open subgroup. Then
\begin{enumerate}
\item $G$ admits a Borel or clopen translation-like action by $\Z$ if and only if $G$ is at most two-ended.
\item $G$ admits a Borel or clopen translation-like action by $F_d$, for all $d\ge2$ if and only if $G$ is non amenable.
\end{enumerate}
\end{theorem}

\begin{proof}
(1) If $G$ is at most two-ended, the statement is covered by the proof of \ref{ThmMain}. Indeed, the translation-like action of $\Z$ on the Cayley-Abels graph $X$ given by \cite[Corollary 3.4]{Seward} used to produce our clopen translation-like action on $G$ can be chosen transitive. In that case it is easily seen that the compact open subgroup $K$ is a compact fundamental domain of the action.
For the converse direction, let $\varphi\in\borel{G}$ be an element generating a cocompact action of $\Z$ on $G$ and let $F$ be a finite set such that $\varphi(g)\in gF$ for all $g\in G$. Let also $K$ be a compact subset such that the translates $\varphi^{(n)}(K)$ cover $G$. Then $K\cup F$ is a compact generating set of $G$ and the number of ends of $G$ corresponds to the number of ends of the Cayley graph of $G$
\footnote{Be careful that this Cayley graph is not locally finite, so the number of ends is the supremum of the number of connected components of $G\setminus B$ when $B$ runs over all \emph{bounded} subsets of $G$ and not over all finite subsets of $G$. See \cite[8.B.12]{CornulierdelaHarpe}) for a detailed treatment. } with respect to $K\cup F$. In this Cayley graph, any orbit map of the action induces a bi-infinite almost-surjective path, in the sense that every point of the space is at uniformly bounded distance from the path. It follows that the graph has at most as many ends as $\Z$, namely two.

(2) Most of the statement is covered by Schneider's Theorem \ref{Schneider}. We only need to show how the action can be chosen cocompact. To do so, we strongly rely on results and notations in \ref{Schneider}. Let $X$ be a Cayley-Abels graph of $G$ and let $K$ be the stabilizer of a point $x_0\in X$. Since $G$ is non amenable and $K$ is compact, the action of $G$ onto $X$ is non-amenable. It follows from \cite[Prop. 2.1 and Thm 2.2]{Schneider} that the Cayley-Abels graph $X$ is covered by the bilipschitz image of a $k$-regular forest for every $k\ge 3$. By extending this forest into a spanning tree, we get a spanning tree in which all vertices have uniformly bounded valency at least 3. According to \cite[Thm 5.5]{Seward}, this tree is bilipschitz equivalent to a 2d-regular tree. Identifying that 2d-regular tree with the Cayley graph of the free group $F_d$ yields a transitive translation-like action of $F_d$ on the Cayley-Abels graph $X$. Again, by Proposition \ref{CayleyAbelsLift} this lifts to a clopen TL-action on $G$ which is clearly cocompact. 
\end{proof}

\begin{remark}
This theorem cannot be extended in full generality to all locally compact groups. Indeed, if $G$ is a connected group then $\clopen{G}\cong G$ and there exists a clopen TL-action of $\Z$ if and only if $G$ is $\Z$-by-compact. This is equivalent to $G$ having exactly two ends and therefore any one-ended connected group violates the first statement of Theorem \ref{SewardTransitiveLC}.

However we were unable to find groups not admitting Borel TL-actions. In particular, we don't know whether the $ax+b$ group, namely the affine group of the real line admits a cocompact Borel TL-action of $\Z$.
\end{remark}

%%%%%%%% Bibliography %%%%%%%% 

\bibliography{BibliographyBurnside}
\bibliographystyle{amsalpha}

\end{document}